%% file: Visible_Networks.tex
\documentclass[11pt]{amsart}

\usepackage{latexsym}
\usepackage{amsmath}
\usepackage{amssymb}

\usepackage{epsfig}
\usepackage{epic}

\usepackage{enumerate}
\usepackage{algorithm}

\usepackage[noend]{algpseudocode2}
\algrenewcommand\algorithmicindent{1em}

\newtheorem{theorem}{Theorem}[section]
\newtheorem{lemma}[theorem]{Lemma}

\newcommand{\cN}{{\mathcal N}}

\newcommand{\cT}{{\mathcal T}}

\begin{document}

\title{Reticulation-Visible Networks}

\author{Magnus Bordewich}
\address{School of Engineering Computer Sciences, Durham University, Durham DH1 3LE, United Kingdom}
\email{m.j.r.bordewich@durham.ac.uk}

\author{Charles Semple}
\address{Biomathematics Research Centre, School of Mathematics and Statistics, University of Canterbury, Christchurch, New Zealand}
\email{charles.semple@canterbury.ac.nz}

\thanks{The second author was supported by the Allan Wilson Centre, and the New Zealand Marsden Fund.}



\date{\today}

\begin{abstract}
Let $X$ be a finite set, $\cN$ be a reticulation-visible network on $X$, and $\cT$ be a rooted binary phylogenetic tree. We show that there is a polynomial-time algorithm for deciding whether or not $\cN$ displays $\cT$. Furthermore, for all $|X|\ge 1$, we show that $\cN$ has at most $8|X|-7$ vertices in total and at most $3|X|-3$ reticulation vertices, and that these upper bounds are sharp.
\end{abstract}

\maketitle

\section{Introduction}

Phylogenetic networks have become increasingly more prominent in the literature as they correctly allow the evolution of certain collections of present-day species to be described with reticulation (non-tree-like) events. However, the evolution of a particular gene can generally be described without reticulation events. As a result, analysing the tree-like information in a phylogenetic network has become a common task. Central to this task is that of deciding if a given phylogenetic network $\cN$ infers a given rooted binary phylogenetic tree on the same collection of taxa. In this paper, we show that if $\cN$ is a so-called reticulation-visible network, then there is a polynomial-time algorithm for making this decision. This resolves a problem left open in~\cite{gam15} and~\cite{ier10}. In the rest of the introduction, we formally state this result as well as the other main result which concerns the number of vertices in a reticulation-visible network.

Throughout the paper, $X$ denotes a non-empty finite set. For $|X|\ge 2$, a {\em phylogenetic network on $X$} is a rooted acyclic digraph with no parallel arcs and the following properties:
\begin{enumerate}[(i)]
\item the root has out-degree two;

\item vertices of out-degree zero have in-degree one, and the set of vertices with out-degree zero is $X$; and

\item all other vertices either have in-degree one and out-degree two, or in-degree two and out-degree one.
\end{enumerate}
If $|X|=1$, then $\cN$ consists of the single vertex in $X$. The vertices in $\cN$ of out-degree zero are called {\em leaves}. Furthermore, the vertices of $\cN$ with in-degree two and out-degree one are called {\em reticulations}, while the vertices of in-degree one and out-degree two are called {\em tree vertices}. The arcs directed into a reticulation are {\em reticulation arcs}; all other arcs are called {\em tree arcs}. Note that, what we have called a phylogenetic network is sometimes referred to as a {\em binary} phylogenetic network. A {\em rooted binary phylogenetic $X$-tree} is a phylogenetic network on $X$ with no reticulations.

\begin{figure}
\center
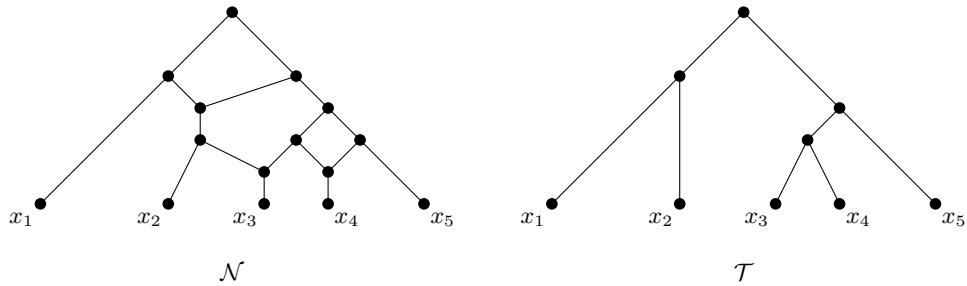
\caption{The phylogenetic network $\cN$ displays the rooted binary phylogenetic tree $\cT$.}
\label{display}
\end{figure}

Let $\cN$ be a phylogenetic network on $X$ and let $\cT$ be a rooted binary phylogenetic $X$-tree. We say that $\cN$ {\em displays} $\cT$ if $\cT$ can be obtained from $\cN$ by deleting arcs and vertices, and contracting degree-two vertices. To illustrate, in Fig.~\ref{display}, the phylogenetic network $\cN$ on $X=\{x_1, x_2, x_3, x_4, x_5\}$ displays the rooted binary phylogenetic $X$-tree $\cT$. The particular problem of interest is the following:

\noindent {\sc Tree Containment} \\
\noindent {\em Instance:} A phylogenetic network $\cN$ on $X$ and a rooted binary phylogenetic $X$-tree $\cT$. \\
\noindent {\em Question:} Does $\cN$ display $\cT$?

\noindent In general, {\sc Tree Containment} is NP-complete~\cite{kan08} even when the instance is highly constrained~\cite{ier10}.

Let $\cN$ be a phylogenetic network on $X$ with root $\rho$. A vertex $v$ in $\cN$ is {\em visible} if there is a leaf $\ell$ in $X$ with the property that every directed path from the root to $\ell$ traverses $v$, in which case, we say {\em $\ell$ verifies the visibility of $v$} (or, more briefly, {\em $\ell$ verifies $v$}). If every reticulation in $\cN$ is visible, then $\cN$ is a {\em reticulation-visible network}. In Fig.~\ref{display}, $\cN$ is a reticulation-visible network. For example, the visibility of the leftmost reticulation is verified by $x_2$. Observe that this reticulation is not verified by $x_3$ as there is a path from the root of $\cN$ to $x_3$ that avoids it. For the reader familiar with tree-child networks, $\cN$ is tree-child network if and only if every vertex in $\cN$ is visible~\cite[Lemma~2]{car09}. Thus tree-child networks are a proper subclass of reticulation-visible networks. It was shown in~\cite{ier10} that there exists a polynomial-time algorithm for {\sc Tree Containment} if $\cN$ is a tree-child network. Here we generalise that result to reticulation-visible networks. The next theorem is the first main result of the paper.

\begin{theorem}
Let $\cN$ be a phylogenetic network on $X$ and let $\cT$ be a rooted binary phylogenetic $X$-tree. Then {\sc Tree Containment} for $\cN$ and $\cT$ can be decided in polynomial time.
\label{decision}
\end{theorem}

It is shown in~\cite{gam15} that a reticulation-visible network on $X$ has at most $4(|X|-1)$ reticulations. The second of the two main results sharpens this result.

\begin{theorem}
Let $\cN$ be a reticulation-visible network $\cN$ on $X$ and let $m=|X|$. Then $\cN$ has at most $8m-7$ vertices in total and at most $3m-3$ reticulations. Moreover, these bounds are sharp for all integers $m\ge 1$.
\label{sharp}
\end{theorem}

The paper is organised as follows. The next section consists of concepts that will be used in the description of the algorithm that establishes Theorem~\ref{decision}. Called {\sc TreeDetection}, the description of this algorithm and the proof of its correctness is given in Section~\ref{algorithm}. In Section~\ref{running}, we analyse the algorithm's running time and show that {\sc TreeDetection} completes in $O(|X|^3)$ steps. Together Sections~\ref{algorithm} and~\ref{running} constitute the proof of Theorem~\ref{decision}. The last section, Section~\ref{visible}, contains the proof of Theorem~\ref{sharp}. Throughout the paper, notation and terminology follows Semple and Steel~\cite{sem03}.

\section{Preliminaries}
\label{prelim}

%

Let $\cN$ be a phylogenetic network, and let $u$ and $v$ be vertices in $\cN$. If $u$ and $v$ are joined by an arc $(u, v)$, we say $u$ is a \emph{parent} of $v$ and, conversely, $v$ is a \emph{child} of $u$. More generally, if $u$ and $v$ are joined by a directed path from $u$ to $v$, we say $u$ is an \emph{ancestor} of $v$ and, conversely, $v$ is a \emph{descendant} of $u$. Furthermore, if $u$ is neither an ancestor nor descendant of $v$, we say $u$ and $v$ are \emph{non-comparable}. A \emph{backward path} in $\cN$ from $v$ to $u$ is an underlying path $v=w_0, w_1, w_2, \ldots, w_{k-1}, w_k=u$ such that, for all $i\in \{1, 2, \ldots, k\}$, we have that $\cN$ contains the arc $(w_i, w_{i-1})$. Observe that if there is a backward path in $\cN$ from $v$ to $u$, then $u$ is an ancestor of $v$. An \emph{up-down path} $P$ in $\cN$ from $v$ to $u$ is an underlying path $u=w_0, w_1, w_2, \ldots, w_{k-1}, w_k=v$ such that, for some $i\le k-1$, we have that $v=w_0, w_1, w_2, \ldots, w_i$ is a backward path from $v$ to $w_i$ and $w_i, w_{i+1}, \ldots, w_k=u$ is a (directed) path from $w_i$ to $v$. The vertex $w_i$ is the \emph{peak} of $P$. Also, a {\em tree path} in $\cN$ from $u$ to $v$ is a (directed) path such that, except possibly $u$, every vertex on the path is either a tree vertex or a leaf. Lastly, the {\em length} of a (directed) path, a backward path, an up-down path, and a tree path is the number of edges in the path.

Let $\cN$ be a phylogenetic network on $X$. A $2$-element subset $\{x, y\}$ of $X$ is a \emph{cherry} in $\cN$ if there is an up-down path of length two between $x$ and $y$. Equivalently, $\{x, y\}$ is a cherry if the parent of $x$ and the parent of $y$ are the same. For a cherry $\{x, y\}$ in $\cN$, let $\cN'$ be obtained from $\cN$ by deleting $x$ and $y$, and their incident arcs, and labelling their common parent (now itself a leaf) with an element not in $X$. We say that $\cN'$ has been obtained from $\cN$ by {\em reducing the cherry $\{x, y\}$}.

Let $\cN$ be a phylogenetic network on $X$ and let $\cT$ be a rooted binary phylogenetic $X$-tree. Suppose that $\cN$ displays $\cT$. Then there is a subgraph $\cT'$ of $\cN$ that is a subdivision of $\cT$. We say $\cT'$ is an {\em embedding} of $\cT$ in $\cN$. Observe that any embedding of $\cT$ in $\cN$ can be formed by deleting exactly one incoming arc at each reticulation and deleting any resulting degree-one vertex that is not a leaf of $\cN$. We refer to the action of deleting one of the two incoming arcs at a reticulation as \emph{resolving the reticulation}.

\section{The Algorithm}
\label{algorithm}

In this section, we present the algorithm {\sc TreeDetection}. This algorithm takes as input a reticulation-visible network $\cN$ on $X$ and a rooted binary phylogenetic $X$-tree $\cT$ and, as we establish in this section, outputs {\bf Yes} if $\cN$ displays $\cT$ and {\bf No} if $\cN$ does not display $\cT$. We begin with some further preliminaries.

Let $\cN$ be a phylogenetic network on $X$ with root $\rho$. Let $a$ and $b$ be distinct elements in $X$. We define the vertex $v_a$ of $\cN$ to be the reticulation at minimum path length from $\rho$ such that $a$ verifies $v_a$ and no other element of $X$ verifies $v_a$. If there is no such reticulation, we define $v_a$ to be $a$. Furthermore, we define $\rho_{ab}$ to be the vertex at maximum path length from $\rho$ such that both $a$ and $b$ verify $\rho_{ab}$. Note that $v_a$ and $\rho_{ab}$ are both well defined since if $A\subseteq X$ and each element of $A$ verifies vertices $u$ and $v$ of $\cN$, then it must be that either $u$ is an ancestor of $v$ or $v$ is an ancestor of $u$.

Briefly, {\sc TreeDetection} proceeds by picking a cherry $\{a, b\}$ of the targeted rooted binary phylogenetic tree $\cT$ and then considering how the leaves $a$ and $b$ are related in $\cN$. There are various cases to consider, but in most cases we can either declare \textbf{No} directly, or we find an arc of $\cN$ that can be deleted so that the resulting phylogenetic network displays $\cT$ if and only if $\cN$ displays $\cT$. In the remaining cases, we can delete the leaf $b$ from both $\cN$ and $\cT$ so that the resulting phylogenetic network displays the resulting rooted binary phylogenetic tree if and only if $\cN$ displays $\cT$. In any case, if we cannot immediately declare \textbf{No}, we reduce the size of $\cN$ by either a vertex or an arc and, by iterating the procedure, we eventually reduce the problem to the trivial case $|X|=1$.

The algorithm {\sc TreeDetection} consists of a number of subroutines. These subroutines are partitioned into three types. The first type consists of $7$ subroutines which we refer to as the {\em easy cases}. The other two types are referred to as {\em special cases}. The second type consists of $5$ subroutines and the third type consists of $4$ subroutines. We first give a top-level description of {\sc TreeDetection} before detailing the cases and their subroutines.

\begin{algorithm}[H]
 \caption{\textsc{TreeDetection}$(\cN, \cT)$}
\begin{algorithmic}[1]
 \Statex\textbf{Input:} A reticulation-visible network $\cN$ on $X$ and a rooted binary phylogenetic $X$-tree $\cT$.
 \Statex\textbf{Output:} \textbf{Yes} if $\cN$ displays $\cT$, and \textbf{No} otherwise.
      \If{$|X|=1$}
      \Statel{Return {\bf Yes}}
      \Else
      \Statel{Find a cherry $\{a, b\}$ in $\cT$ at maximum path length from the root.}
      \For{$i =1$ to $7$}
      \If{Easy Case $i$ applies}
      \Statel{Execute subroutine {\sc Easy Case $i$}.}
      \Statel{Halt.}
      \EndIf
      \EndFor
      \If{the sibling of the parent of $a$ and $b$ is a single leaf $c$}
      \For{$i =1$ to $4$}
      \If{Special Case $1.i$ applies}
      \Statel{Execute subroutine {\sc Special Case $1.i$}.}
      \Statel{Halt.}
      \EndIf
      \EndFor
	\Statel{Execute subroutine {\sc Special Case $1.5$}.}
      \Else { the sibling of the parent of $a$ and $b$ is the parent of a cherry $\{c, d\}$}
      \For{$i =1$ to $3$}
      \If{Special Case $2.i$ applies}
      \Statel{Execute subroutine {\sc Special Case $2.i$}.}
      \Statel{Halt.}
      \EndIf
      \EndFor
	\Statel{Execute subroutine {\sc Special Case $2.4$}.}
       \EndIf
      \EndIf
 \end{algorithmic}
\end{algorithm}

With the top-level description of the algorithm established, we now turn to the details of the various cases.
If $|X|>1$, let $\{a, b\}$ be a cherry in $\cT$ whose distance from its root is maximised. We then have the following `easy' cases (each one holding only if the preceeding cases do not hold). Subroutines for each easy case as well as their justifications are given in Section~\ref{easycases}.

\noindent\textsc{Easy Cases}
\begin{itemize}
\item[(EC1)] The $2$-element subset $\{a, b\}$ is a cherry in $\cN$.

\item[(EC2)] For some $i\in \{a, b\}$, we have $v_i = i$.

\item[(EC3)] For $\{i, j\}=\{a, b\}$, we have $v_i$ is an ancestor of $v_j$.

\item[(EC4)] For $\{i, j\}=\{a,b\}$, there is an ancestor $w$ of $v_i$ that is verified by $i$ but not by $j$.

\item[(EC5)] For some $i\in \{a,b\}$, there is a descendant $u$ of $v_i$ that is a reticulation not verified by $i$.

\item[(EC6)] For $\{i, j\}=\{a, b\}$, there is a reticulation $u$ that is an ancestor of $v_i$ but not of $j$.

\item[(EC7)] For some $i\in \{i, j\}$, there are two non-comparable reticulations $u$ and $w$ that are descendants of $\rho_{ab}$ and ancestors of $v_i$.
\end{itemize}

\begin{figure}
\center
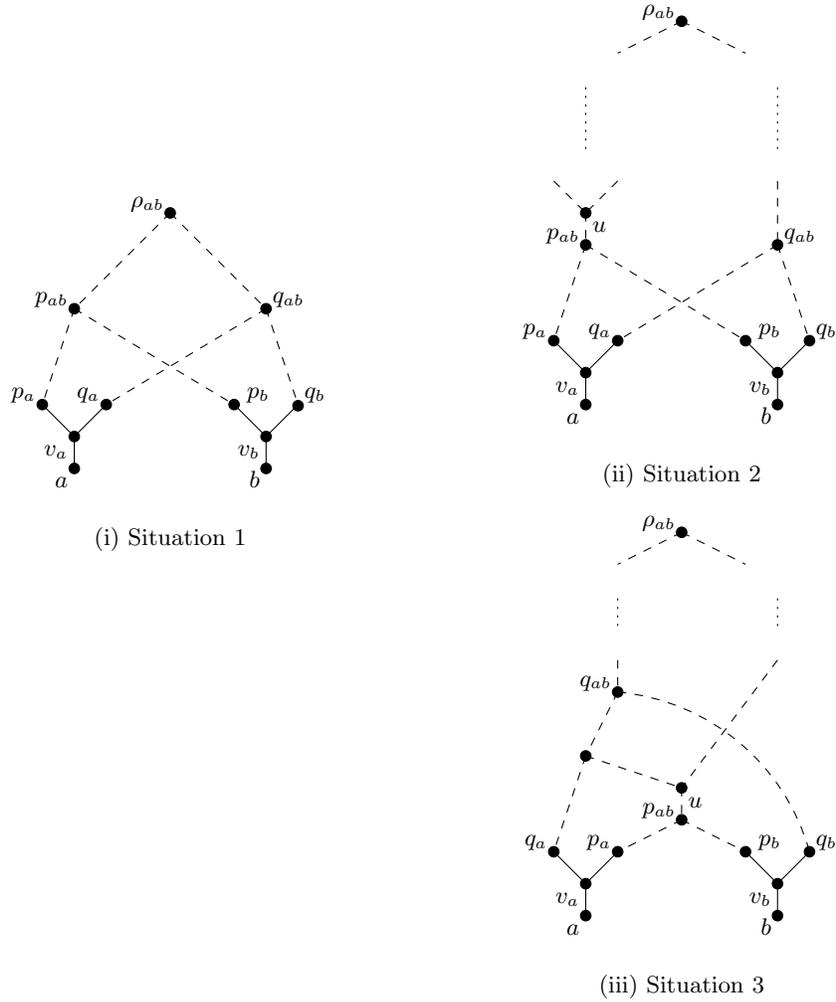
\caption{Situations 1, 2, and 3.}
\label{situations}
\end{figure}

We establish in Lemma~\ref{inspecialcase} below that if we are not in one of (EC1)--(EC7), then we must be in one of the three situations illustrated in Fig.~\ref{situations}. Here, Figs~\ref{situations}(i)--(iii) is a subgraph of the phylogenetic network at the completion of (EC1)--(EC7). More precisely, each of these subgraphs consists of all (directed) paths from $\rho_{ab}$ to either $a$ or $b$. Solid lines correspond to edges. Dashed lines between named vertices correspond to (directed) paths with no reticulations unless one of the named vertices is $u$. Note that $v_a$ and $v_b$ are the parents of $a$ and $b$, respectively. In Situations~2 and~3, the structure of the subgraph immediately `below' $\rho_{ab}$ is not completely determined. Two key vertices in the analysis to follow are those labelled $p_{ab}$ and $ q_{ab}$ in Fig.~\ref{situations}. If there is a reticulation on any path between $\rho_{ab}$ and $v_a$, then there is a unqiue maximal reticulation $u$ by (EC7). In this case we define $p_a$ to be the parent of $v_a$ on the path from $u$ to $v_a$ and $q_a$ to be the other parent of $v_a$. If there is no reticulation, let the children of $\rho_{ab}$ be $p$ and $q$ and define $p_a$ to be the parent of $v_a$ on the path from $p$ to $v_a$, $q_a$ to be the parent of $v_a$ on the path from $q$ to $v_a$. In each case define $p_b,q_b$ likewise. Then $p_{ab}$ is the ancestor of $p_a$ and $p_b$ at maximal distance from $\rho_{ab}$ and  $q_{ab}$ is the ancestor of $q_a$ and $q_b$ at maximal distance from $\rho_{ab}$. (The proof of Lemma~\ref{inspecialcase}  shows that these are well defined.)  

For each situation, the crucial feature we will need to consider is the structure outside of the cherry $\{a, b\}$ in $\cT$. Since $\{a, b\}$ was chosen to be at maximal distance from the root, the sibling of the parent of $a$ and $b$ in $\cT$ is either a leaf $c$ or the parent of leaves $c$ and $d$, where $\{c, d\}$ is a cherry.
First, we consider the case that the sibling is a single leaf $c$. Let $v_c$ be the reticulation in $\cN$ at minimum path length from the root such that $c$ verifies $v_c$ and no other element verifies $c$. If there is no such reticulation, define $v_c$ to be $c$.  Let $\hat v_c$ be the first reticulation on a path from $v_c$ to $c$ (including $v_c$ itself) that is not an ancestor of $a$ or $b$. Note that $\hat v_c=c$ if no such vertex exists. Furthermore, $c$ verifies $\hat v_c$, otherwise some other leaf would verify it, and this leaf would also verify $v_c$. The first of the special cases is detailed below. The corresponding subroutines and their justifications are given in Section~\ref{specialcasealgs}.

\noindent\textsc{Special Case: Single Leaf Sibling}
\begin {itemize}
\item[(SC1.1)] The leaf $c$ is not a descendant of $\rho_{ab}$.

\item[(SC1.2)] There is a descendant $u$ of $\hat v_c$ that is a reticulation and not verified by $c$.

\item[(SC1.3)] There is an ancestor $u$ of $\hat v_c$ that is a reticulation and not an ancestor of $a$ or $b$.

\item[(SC1.4)] The leaf $c$ does not verify $\rho_{ab}$.
\end{itemize}

The fifth subroutine deals with the case when none of (SC1.1)--(SC1.4) hold, in which case, $\hat{v}_c\neq c$. Denote the parents of $\hat v_c$ as $p_c$ and $q_c$. Let $p'_c$ be the unique  vertex at maximal distance from $\rho_{ab}$ that is an ancestor of $p_c$ and an ancestor of either $a$ or $b$. Likewise, let $q'_c$ be the unique vertex at maximal distance from $\rho_{ab}$ that is an ancestor of $q_c$ and an ancestor of either $a$ or $b$. For each of $p'_c$ and $q'_c$ either it is an ancestor of $p_{ab}$ or an ancestor of $q_{ab}$, or it lies on one of the following (directed) paths: $p_{ab}$ to $p_a$, call it $P_1$; $p_{ab}$ to $p_b$, call it $P_2$; $q_{ab}$ to $q_a$, call it $P_3$;  and $q_{ab}$ to $q_b$, call it $P_4$. With these definitions, we identify an arc to delete from $\cN$, or a vertex to delete from both $\cT$ and $\cN$. The appropriate deletions are given in Table~\ref{singlectable}, while justification for these deletions is given in Section~\ref{specialcasealgs}. In the table, we denote that $p'_c$ (respectively, $q'_c$) is an ancestor of a vertex $u$ by $>u$.

\begin{table}[H]
\caption{Table showing the actions to perform on $\cN$ and $\cT$ to create $\cN'$ and $\cT'$ depending on the locations of $p'_c$ and $q'_c$ in each of Situations 1--3 when the sibling of the parent of $a$ and $b$ is a single leaf $c$. If the entry corresponding to the subcase we are in contains a single arc $e$, then $\cN'=\cN\backslash e$ and $\cT'=\cT$. If the entry contains the leaf $b$, then $\cN'=\cN\backslash b$ and $\cT'=\cT\backslash b$.}
\begin{center}
\begin{tabular}{|c|c|c|c|c|} \hline
$p'_c$ & $q'_c$ & Sit.\ 1 & Sit.\ 2 & Sit.\ 3 \\
\hline
$P_1$ or $P_2$ & $P_1$ or $P_2$ & $(p_a, v_a)$ & $(p_a, v_a)$ & $(p_a, v_a)$ \\
$P_1$ or $P_2$ & $P_3$ or $P_4$ & $(p_c, \hat v_c)$ & $(q_c, \hat v_c)$ & $(q_c, \hat v_c)$ \\
$P_1$ or $P_2$ & $>p_{ab}$ & $(p_c, \hat v_c)$ & $(p_c, \hat v_c)$ & $(p_c, \hat v_c)$ \\
$P_1$ or $P_2$ & $>q_{ab}$ & $(p_a, v_a)$ & $(p_a, v_a)$ & $(p_a, v_a)$ \\
$P_3$ or $P_4$ & $P_1$ or $P_2$ & $(p_c, \hat v_c)$ & $(p_c, \hat v_c)$ & $(p_c, \hat v_c)$ \\
$P_3$ or $P_4$ & $P_3$ or $P_4$ & $(q_a, v_a)$ & $(q_a, v_a)$ & $(q_a, v_a)$ \\
$P_3$ or $P_4$ & $>p_{ab}$ & $(q_a, v_a)$ & $(p_c, \hat v_c)$ & $(p_c, \hat v_c)$ \\
$P_3$ or $P_4$ & $>q_{ab}$ & $(p_c, \hat v_c)$ & $(p_c, \hat v_c)$ & $(p_c, \hat v_c)$ \\
$>p_{ab}$ & $P_1$ or $P_2$ & $(q_c, \hat v_c)$ & $(q_c, \hat v_c)$ & $(q_c, \hat v_c)$ \\
$>p_{ab}$& $P_3$ or $P_4$ & $(q_a, v_a)$ &$(q_c, \hat v_c)$ & $(q_c, \hat v_c)$ \\
$>p_{ab}$ & $>p_{ab}$ & $(q_a, v_a)$ & $(q_a, v_a)$ & $(q_a, v_a)$ \\
$>p_{ab}$ & $>q_{ab}$ & $b$ & $b$ & $b$ \\
$>q_{ab}$ & $P_1$ or $P_2$ & $(p_a, v_a)$ & $(p_a, v_a)$ & $(p_a, v_a)$ \\
$>q_{ab}$& $P_3$ or $P_4$ & $(q_c, \hat v_c)$ & $(q_c, \hat v_c)$ & $(q_c, \hat v_c)$ \\
$>q_{ab}$ & $>p_{ab}$ & $b$ & $b$ & $b$ \\
$>q_{ab}$ & $>q_{ab}$ & $(p_a, v_a)$ & $(p_a, v_a)$ & $(p_a, v_a)$ \\ \hline
\end{tabular}
\end{center}
\label{singlectable}
\end{table}

Finally, we consider the case that the sibling of the parent of $a$ and $b$ in $\cT$ is the parent of leaves $c$ and $d$, where $\{c, d\}$ is a cherry. Since $\{c, d\}$ is also a cherry at maximum distance from the root of $\cT$, we first check whether any of (EC1)--(EC7) apply when considering $\{c, d\}$ instead of $\{a, b\}$. If so, we make the appropriate reduction of $\cN$; otherwise, we may assume that the configuration of $\{c, d\}$ in $\cN$ is analogous to one of Situations~1--3. In particular, $v_c$ is the parent of $c$ and $v_d$ is the parent of $d$. Furthermore, without loss of generality, we may assume that $\rho_{ab}$ is an ancestor of (or equal to) $\rho_{cd}$ or that they are non-comparable (if this is not the case, a simple relabelling will resolve the issue). The second special case is detailed below. The corresponding subroutines and their justifications are given in Section~\ref{specialcasealgs}.

\noindent\textsc{Special Case: Cherry Sibling}
\begin{itemize}
\item[(SC2.1)] The vertices $\rho_{cd}$ and $\rho_{ab}$ are non-comparable.

\item[(SC2.2)] The vertex $\rho_{cd}$ is not on one of the paths from $\rho_{ab}$ to either $v_a$ or $v_b$.

\item[(SC2.3)] There is a reticulation $u$ that is an ancestor of $p_{cd}$ or $q_{cd}$ but it is not an ancestor of $a$ or $b$.
\end{itemize}

Now we consider the locations of $p_{cd}$ and $q_{cd}$ relative to $a$ and $b$. Since $\rho_{cd}$ is a descendant of $\rho_{ab}$ (or $\rho_{cd}=\rho_{ab}$), each of $p_{cd}$ and $q_{cd}$ must either lie on a path from $\rho_{ab}$ to either $v_a$ or $v_b$, or have a unique ancestor on such a path a maximal distance from $\rho_{ab}$. Let $p'_{cd}$ be this unique ancestor of $p_{cd}$, or equal to $p_{cd}$ if it lies on such a path itself, and let $q'_{cd}$ be the analogous vertex for $q_{cd}$. We have a similar table to that of the previous special case, but this time we must distinguish between subcases where $p'_{cd}$ is or is not equal to $p_{cd}$, and $q'_{cd}$ is or is not equal to $q_{cd}$. We use following definitions:
\begin{enumerate}[(i)]
\item Let $p'_{cd}$ (respectively, $q'_{cd}$) be of {\em Type 1} if it is a descendant of $p_{ab}$ or an ancestor of $p_{ab}$ and $p'_{cd}=p_{cd}$ (respectively, $q'_{cd}= q_{cd}$).

\item Let $p'_{cd}$ (respectively, $q'_{cd}$) be of {\em Type 2} if it is a descendant of $q_{ab}$ or an ancestor of $q_{ab}$ and $p'_{cd}=p_{cd}$ (respectively, $q'_{cd}= q_{cd}$).

\item Let $p'_{cd}$ (respectively, $q'_{cd}$) be of {\em Type 3} if it is an ancestor of $p_{ab}$ and $p'_{cd}\neq p_{cd}$ (respectively, $q'_{cd}\neq q_{cd}$).

\item Let $p'_{cd}$ (respectively, $q'_{cd}$) be of {\em Type 4} if it is an ancestor of $q_{ab}$ and $p'_{cd}\neq p_{cd}$ (respectively, $q'_{cd}\neq q_{cd}$).
\end{enumerate}

\begin{table}[H]
\caption{Table showing the actions to perform on $\cN$ and $\cT$ to form $\cN'$ and $\cT'$, respectively, depending on the locations of $p_{cd}$ and $q_{cd}$ in each of Situations~1--3 in the case that the sibling of the parent of $a$ and $b$ realises a cherry $\{c, d\}$. If the entry corresponding to the subcase we are in contains a single arc $e$, then $\cN'=\cN\backslash e$ and $\cT'=\cT$. If the entry contains the leaf $b$, then $\cN'=\cN\backslash b$ and $\cT'=\cT\backslash b$.}
\begin{center}
\begin{tabular}{|c|c|c|c|c|} \hline
$p'_{cd}$ & $q'_{cd}$ & Sit.\ 1 & Sit.\ 2 & Sit.\ 3 \\
\hline
Type 1 & Type 1 & $(p_a, v_a)$ & $(p_a, v_a)$ & $(p_a, v_a)$ \\
Type 1 & Type 2 & $(p_c, v_c)$ & $(q_c, v_c)$ & $(q_c, v_c)$ \\
Type 1 &  Type 3 & $(p_c, v_c)$ & $(p_c, v_c)$ & $(p_c, v_c)$ \\
Type 1 &  Type 4 & $(p_a, v_a)$ & $(p_a, v_a)$ & $(p_a, v_a)$ \\
 Type 2 & Type 1 & $(p_c, v_c)$ & $(p_c, v_c)$ & $(p_c, v_c)$ \\
 Type 2 &  Type 2 & $(q_a, v_a)$ & $(q_a, v_a)$ & $(q_a, v_a)$ \\
 Type 2 &  Type 3 & $(q_a, v_a)$ & $(p_c, v_c)$ & $(p_c, v_c)$ \\
 Type 2 &  Type 4 & $(p_c, v_c)$ & $(p_c, v_c)$ & $(p_c, v_c)$ \\
 Type 3 &  Type 1 & $(q_c, v_c)$ & $(q_c, v_c)$ & $(q_c, v_c)$ \\
 Type 3 &  Type 2 & $(q_a, v_a)$ & $(q_c, v_c)$ & $(q_c, v_c)$ \\
 Type 3 &  Type 3 & $(q_a, v_a)$ & $(q_a, v_a)$ & $(q_a, v_a)$ \\
 Type 3 & Type 4 & $b$ & $b$ & $b$ \\
 Type 4 &  Type 1 & $(p_a, v_a)$ & $(p_a, v_a)$ & $(p_a, v_a)$ \\
 Type 4&  Type 2 & $(q_c, v_c)$ & $(q_c, v_c)$ & $(q_c, v_c)$ \\
 Type 4 &  Type 3 & $b$ & $b$ & $b$ \\
 Type 4 &  Type 4 & $(p_a, v_a)$ & $(p_a, v_a)$ & $(p_a, v_a)$ \\ \hline
\end{tabular}
\end{center}
\label{cherryctable}
\end{table}

\subsection{Easy Cases}
\label{easycases}
\setcounter{algorithm}{0}

Here we present the subroutines for each of (EC1)--(EC7) with justification following each subroutine. Throughout the presentation of these subroutines, as well as those presented for the two special cases, whenever we delete an arc or a leaf from $\cN$ or $\cT$, we additionally contract any resulting degree-two vertex.

\begin{algorithm}[H]
\floatname{algorithm}{Subroutine}
 \caption{\textsc{Easy Case 1}}
\begin{algorithmic}[1]
  \Statel{Delete $b$ in $\cN$ and $\cT$ to obtain $\cN'$ and $\cT'$, respectively}
  \Statel{Return \textsc{TreeDetection$(\cN', \cT')$}.}
 \end{algorithmic}
\end{algorithm}
Suppose $\{a, b\}$ is a cherry in $\cN$. Then any embedding of $\cT'$ in $\cN'$ can be extended to an embedding of $\cT$ in $\cN$ by appending $b$ to the arc incident with $a$ in both $\cN'$ and $\cT'$, and including the new edge in the embedding. Conversely, any embedding of $\cT$ in $\cN$ gives rise to an embedding of $\cT'$ in $\cN'$ by deleting $b$, its incident arc, and contracting the resulting degree-two vertex.

\begin{algorithm}[H]
\floatname{algorithm}{Subroutine}
 \caption{\textsc{Easy Case 2}}
\begin{algorithmic}[1]
 \Statex\textbf{Requirement:} Easy Cases 1 does not apply, and $v_i=i$ for some $i\in \{a, b\}$.
     \Statel{Let $p_i$ denote the parent of $i$.}
      \If{there is a tree path from $p_i$ to a leaf $\ell\not\in \{a, b\}$,} \Statel{Return \textbf{No}.}\label{AlgStep5a}
      \Else { there is a tree path from $p_i$ to a parent $p_u$ of a reticulation $u$.}\label{AlgStep5b}
       \If{$u\neq v_j$,}
        \Statel{Construct $\cN'$ from $\cN$ by deleting the arc $(p_u, u)$.}
       \Else { $u=v_j$}
        \Statel{Construct $\cN'$ from $\cN$ by deleting the incoming arc to $u$ that is not $(p_u, u)$.}
       \EndIf
               \Statel{Return \textsc{TreeDetection$(\cN', \cT)$.}}\label{AlgStep5bii}
                \EndIf
 \end{algorithmic}
\end{algorithm}
%

If there is a tree path from $p_i$ to a leaf $\ell\not\in \{a, b\}$, then every rooted binary phylogenetic $X$-tree  embedded in $\cN$ must display the rooted triple $i\ell|j$. Since $\{i, j\}=\{a, b\}$ is a cherry in $\cT$, no embedding of $\cT$ in $\cN$ can exist. Thus returning {\bf No} in Line~3 is correct.

So assume that there is no such tree path from $p_i$. Then there is a tree path from $p_i$ to a parent $p_u$ of a reticulation $u$. Assume $u\neq v_j$. Observe that $u$ is not a descendant of $v_j$, otherwise $i$ would also verify $v_j$ contradicting the definition of $v_j$. Thus $u$ is verified by some leaf $\ell\not\in \{a, b\}$. Note that we can choose $\ell\neq j$ since $u\neq v_j$ and $u$ is not a descendant of $v_j$. Now an embedding of a rooted binary phylogenetic $X$-tree in $\cN$ that uses arc $(p_u, u)$ would display the rooted triple $i\ell|j$ if $v_j$ is not a descendant of $u$, and $j\ell|i$ if $v_j$ is a descendant of $u$.

On the other hand, if $u=v_j$, then it is easily seen that if there is an embedding of $\cT$ that does not use arc $(p_u, u)$, then there is also an embedding of $\cT$ that uses arc $(p_u, u)$. Thus in both cases there is an embedding of $\cT$ in $\cN$ if and only if there is an embedding of $\cT$ in $\cN'$.

The justifications for the next three subroutines are similar to that for Subroutine 2 and are omitted.

\begin{algorithm}[H]
\floatname{algorithm}{Subroutine}
 \caption{\textsc{Easy Case 3}}
\begin{algorithmic}[1]
 \Statex\textbf{Requirement:} Easy Cases 1--2 do not apply, and $v_i$ is an ancestor of $v_j$ for $\{i, j\}= \{a, b\}$.
       \Statel{Let $p_j$ be a parent of $v_j$ that is a descendant of $v_i$, and let $u$ be the first reticulation reached on a backward path from $p_j$ to $v_i$.}
        \If{$u$ is not an ancestor of $i$,}
         \Statel{Construct $\cN'$ from $\cN$ by deleting the arc $(p_j, v_j)$.}
        \Else { $u$ is an ancestor of $i$}
         \Statel{Construct $\cN'$ from $\cN$ by deleting the incoming arc to $v_j$ that is not $(p_j, v_j)$.}
        \EndIf
         \Statel{Return \textsc{TreeDetection$(\cN', \cT)$}.}
 \end{algorithmic}
\end{algorithm}

\begin{algorithm}[H]
\floatname{algorithm}{Subroutine}
 \caption{\textsc{Easy Case 4}}
\begin{algorithmic}[1]
 \Statex\textbf{Requirement:} Easy Cases 1--3 do not apply and, for some $i\in \{a, b\}$, there is an ancestor $w$ of $v_i$ which is verified by $i$ but not by $j$.
        \If{$w$ is not an ancestor of $j$,}
         \Statel{Construct $\cN'$ from $\cN$ by deleting either incoming arc to $v_i$.}
        \ElsIf{there is a reticulation on a path $P$ from $w$ to $v_j$ that is not an ancestor of $i$,}
         \Statel{Construct $\cN'$ from $\cN$ by deleting the incoming arc to $v_j$ that is on $P$.}
        \Else
         \Statel{Construct $\cN'$ from $\cN$ by deleting the incoming arc to $v_j$ that is not on a backward path to $w$ from $v_j$.}
        \EndIf
        \Statel{Return \textsc{TreeDetection$(\cN', \cT)$}.}
 \end{algorithmic}
\end{algorithm}

\begin{algorithm}[H]
\floatname{algorithm}{Subroutine}
 \caption{\textsc{Easy Case 5}}
\begin{algorithmic}[1]
 \Statex\textbf{Requirement:} Easy Cases 1--4 do not apply and, for some $i\in \{a, b\}$, there is a path $P$ from $v_i$ to a reticulation not verified by $i$.
        \Statel{Let $u$ be the first reticulation on $P$ not verified by $i$ and let $p_u$ be the parent of $u$ on $P$.}
        \Statel{Construct $\cN'$ from $\cN$ by deleting the arc $(p_u, u)$.}
        \Statel{Return \textsc{TreeDetection$(\cN', \cT)$}.}
 \end{algorithmic}
\end{algorithm}

\begin{algorithm}[H]
\floatname{algorithm}{Subroutine}
 \caption{\textsc{Easy Case 6}}
\begin{algorithmic}[1]
 \Statex\textbf{Requirement:} Easy Cases 1--5 do not apply and, for $\{i, j\}=\{a, b\}$, there is a reticulation $u$ that is an ancestor of $v_i$ but not an ancestor of $j$.
        \Statel{Let $p_i$ be a parent of $v_i$ that is a descendant of $u$.}
        \Statel{Construct $\cN'$ from $\cN$ by deleting the arc $(p_i, v_i)$.}
        \Statel{Return \textsc{TreeDetection$(\cN', \cT)$}.}
 \end{algorithmic}
\end{algorithm}
Since $u$ is an ancestor of $v_i$ it must be verified by some leaf $\ell\neq i$. But $u$ is not an ancestor of $j$, so $\ell\neq j$. If an embedding of $\cT$ uses the arc $(p_i,v_i)$, then it must display the rooted triple $i\ell|j$; a contradiction. Thus there is an embedding of $\cT$ in $\cN$ if and only if there is an embedding of $\cT$ in $\cN'$.

\begin{algorithm}[H]
\floatname{algorithm}{Subroutine}
 \caption{\textsc{Easy Case 7}}
\begin{algorithmic}[1]
 \Statex\textbf{Requirement:} Easy Cases 1--6 do not apply, and there are two non-comparable reticulations $u$ and $w$ that are descendants of $\rho_{ab}$ and, for some $i\in \{a, b\}$, ancestors of $v_i$.
        \Statel{Let $p_i$ be a parent of $v_i$ that is a descendant of $u$, and let $q_i$ be the other parent of $v_i$.}
        \Statel{Let $\ell_u$ be a leaf verifying $u$, and let $\ell_w$ be a leaf verifying $w$.}
        \If{there is a reticulation $u'$ that is on a path from $u$ to $v_i$ at minimal distance from $u$}
        \Statel{Let $\ell_u'\not\in \{\ell_u,\ell_w,a, b\}$ be a leaf verifying $u'$.}
                \If{$\cT$ displays the rooted triple $\ell_u'\ell_u|\ell_w$}
        \Statel{Construct $\cN'$ from $\cN$ by deleting the incoming arc  to $u'$ on the path from $w$.}
        \Else
        \Statel{Construct $\cN'$ from $\cN$ by deleting the incoming arc  to $u'$ on the path from $u$.}
        \EndIf
        \Else
        \If{$\cT$ displays the rooted triple $i\ell_u|\ell_w$}
        \Statel{Construct $\cN'$ from $\cN$ by deleting the arc $(q_i, v_i)$.}
        \Else
        \Statel{Construct $\cN'$ from $\cN$ by deleting the arc $(p_i, v_i)$.}
        \EndIf
        \EndIf
                \Statel{Return \textsc{TreeDetection$(\cN', \cT)$}.}
 \end{algorithmic}
\end{algorithm}
Without loss of generality, we may assume $u$ and $w$ are at maximal distance from $\rho_{ab}$, that is, there is no reticulation $u'$ that is a descendant of $u$ such that $u'$ and $w$ satisfy the conditions of the requirement and, likewise, no reticulation $w'$ that is a descendant of $w$ such that $u$ and $w'$ satisfy the conditions of the requirement. Since $u$ and $w$ are descendants of $\rho_{ab}$, neither is verified by both $a$ and $b$ and so, by (EC4), neither is verified by $a$ or by $b$. Thus there are leaves $\ell_u,\ell_w\not\in \{a, b\}$ such that $\ell_u$ verifies $u$ and $\ell_w$ verifies $w$. Since $u$ and $w$ are non-comparable, $\ell_u\neq \ell_w$ as not all paths to a single leaf can go through two non-comparable vertices. By (EC6), each of $u$ and $w$ is an ancestor of both $a$ and $b$. 

Suppose first that there is a reticulation $u'$ that is on a path from $u$ to $v_i$. Take $u'$ to be such a reticulation at shortest distance to $u$. By our assumption, $u'$ cannot be non-comparable with $w$, so it must be a descendant of both $w$ and $u$. Furthermore, $u'$ must be verified by some $\ell_{u'}\not\in \{\ell_u,\ell_w,a, b\}$. Since $\cT$ must display at most one of $\ell_u'\ell_u|\ell_w$ and $\ell_u'\ell_w|\ell_u$; we delete the appropriate reticulation arc incident with $u'$ so that there is an embedding of $\cT$ in $\cN$ if and only if there is an embedding of $\cT$ in $\cN'$.

By symmetry, we may now assume that there is no reticulation that is on a path from $u$ to $v_i$ or from $w$ to $v_i$. Let $p_i$ be the parent of $v_i$ that is a descendant of $u$, and let $q_i$ be the parent of $v_i$ that is a descendant of $w$. Note that these are well defined by our latest assumption and, also, $p_a\neq q_a$ since $u$ and $w$ are non-comparable. If an embedding of $\cT$ uses arc $(p_a,v_a)$, it must display the rooted triple $a\ell_u|\ell_w$, and if it uses arc $(q_a,v_a)$, it must display the rooted triple $a\ell_w|\ell_u$. At most one of these rooted triples is displayed by $\cT$. Thus, with $\cN'$ as constructed, there is an embedding of $\cT$ in $\cN$ if and only if there is an embedding of $\cT$ in $\cN'$.

\subsection{Special Cases}\label{specialcasealgs}
We first show that if none of (EC1)--(EC7) hold, then we are indeed in one of Situations~1, 2, and 3. Thereafter, we present the subroutines and their justifications for the special cases.

\begin{lemma}\label{inspecialcase}
Let $|X|\ge 2$, and let $\cN$ be a phylogenetic network on $X$ and let $\cT$ be a rooted binary phylogenetic $X$-tree. Let $\{a, b\}$ be a cherry in $\cT$ whose distance from its root is maximised. If none of {\rm (EC1)--(EC7)} applies, then one of Situations~$1$, $2$, and $3$ applies.
\end{lemma}

\begin{proof}
Let $i\in \{a, b\}$. Suppose first that the parent of $i$ is a tree vertex $v$. Then there is a tree-path from $v$ that reaches either a parent of a reticulation $u$ not verified by $i$, or a leaf $\ell\neq i$. By (EC2), $v_i\neq i$, so EC5 excludes the tree-path reaching a parent of a reticulation $u$. But then $\ell$ verifies every vertex that $i$ verifies (except $i$ itself), contradicting the definition of $v_i$.

Now suppose that the parent of $i$ is a reticulation $w\neq v_i$. Let the parents of $w$ be $p_w$ and $q_w$, and note that $p_w$ and $q_w$ are both tree vertices since parents of reticulations are not visible. The vertex $v_i$ must be an ancestor of both $p_w$ and $q_w$. By (EC5), $p_w$ cannot have a path to a reticulation other than $w$. But then $p_w$ must have a tree path to a leaf $\ell$ that verifies $v_i$; a contradiction. It now follows that $v_i$ is the parent of $i$ for all $i\in\{a, b\}$.

The vertex $\rho_{ab}$ cannot be a reticulation, as its child would also be verified by $a$ and $b$. Thus $\rho_{ab}$ has two children $p$ and $q$, neither of which is verified by both $a$ and $b$. Furthermore, by (EC4), neither $p$ nor $q$ is verified by exactly one of $a$ and $b$. Hence there is a (directed) path from $p$ to each of $a$ and $b$, and also from $q$ to each of $a$ and $b$.

Except for $v_a$ and $v_b$, suppose that there are no reticulations on any path between $\rho_{ab}$ and $v_a$ and between $\rho_{ab}$ and $v_b$. Then the paths from $p$ to each of $a$ and $b$ are unique, and the paths from $q$ to each of $a$ and $b$ are unique. Denote the parent of $v_a$ on the path from $p$ to $a$ by $p_a$, and the parent of $v_a$ on the path from $q$ to $a$ by $q_a$. Similarly, denote by $p_b$ and $q_b$ the parents of $v_b$. Let $p_{ab}$ be the last vertex on the path from $p$ to $p_a$ that is an ancestor of $b$, and let $q_{ab}$ be the last vertex on the path from $q_{ab}$ to $q_a$ that is an ancestor of $b$. (Possibly $p_{ab}=p$ or $q_{ab}=q$.) Thus, if, apart from $v_a$ and $v_b$, there are no reticulations on any path between $\rho_{ab}$ and $a$, and between $\rho_{ab}$ and $b$, then we are in Situation~1.

Now suppose that, in addition to $v_a$ and $v_b$, there is a reticulation on a path between $\rho_{ab}$ and $v_a$ or between $\rho_{ab}$ and $v_b$. Let $u$ be such a reticulation at maximal distance from $\rho_{ab}$. By (EC6), $u$ is an ancestor of both $a$ and $b$. By (EC7), every other reticulation on a path from $\rho_{ab}$ to $v_a$ or from $\rho_{ab}$ to $v_b$ that is not $v_a$ or $v_b$ is an ancestor of $u$. By maximality, the paths from $u$ to $a$ and from $u$ to $b$ are unique. Let the parent of $v_a$ on the path from $u$ to $a$ be $p_a$ and let the parent of $v_b$ on the path from $u$ to $b$ be $p_b$. Let $p_{ab}$ be the last vertex on the path from $u$ to $p_a$ that is an ancestor of $b$. Let $q_a$ and $q_b$ be the other parents of $v_a$ and $v_b$, respectively. Let $q_{ab}$ be the ancestor of $q_a$ and $q_b$ at maximal distance from $\rho_{ab}$. Note that $q_{ab}$ is well defined. To see this, assume that there were two possibilities $q_{ab}$ and $q'_{ab}$, both ancestors of $q_a$ and $q_b$, and at maximal distance from $\rho_{ab}$. By maximality, $q_{ab}$ and $q'_{ab}$ are non-comparable, so there would have to be a reticulation where the paths from $q_{ab}$ and $q'_{ab}$ to $q_a$ merge, and where the paths from $q_{ab}$ and $q'_{ab}$ to $q_b$ merge. These two reticulations would also have to be non-comparable, contradicting the exclusion of (EC7).

If $q_{ab}$ is not an ancestor of $u$, then we are in Situation~2. If $q_{ab}$ is an ancestor of $u$, then we are in Situation~3.
\end{proof}

We now present the subroutines for the Special Case when the sibling of the parent of $a$ and $b$ in $\cT$ is a single leaf $c$.

\begin{algorithm}[H]\label{SC:1}
\floatname{algorithm}{Subroutine}
 \caption{\textsc{Special Case 1.1}}
\begin{algorithmic}[1]
 \Statex\textbf{Requirement:} The leaf $c$ is not a descendant of $\rho_{ab}$.
        \Statel{Construct $\cN'$ from $\cN$ by deleting the arc $(p_a, v_a)$.}
                \Statel{Return \textsc{TreeDetection$(\cN', \cT)$}.}
 \end{algorithmic}
\end{algorithm}
This is valid since any embedding of $\cT$ in $\cN$ must have the last common ancestor of $a$ and $c$ as an ancestor of $\rho_{ab}$. Thus there are no paths (in the embedding) from $\rho_{ab}$ to any leaves other than $a$ and $b$, so we can resolve $v_a$ either way without changing the result of the embedding. In particular, $\cN$ displays $\cT$ if and only if $\cN'$ displays $\cT$.

The justification of the next subroutine is similar to that given for {\sc Easy Case~2} and is omitted.

\begin{algorithm}[H]
\floatname{algorithm}{Subroutine}
 \caption{\textsc{Special Case 1.2}}
\begin{algorithmic}[1]
 \Statex\textbf{Requirement:} There is a path $P$ from $\hat v_c$ to a reticulation not verified by $c$.
        \Statel{Let $u$ be the first reticulation on $P$ not verified by $c$ and let $p_u$ be the parent of $u$ on $P$.}
        \Statel{Construct $\cN'$ from $\cN$ by deleting the arc $(p_u, u)$.}
        \Statel{Return \textsc{TreeDetection$(\cN', \cT)$}.}
 \end{algorithmic}
\end{algorithm}

\begin{algorithm}[H]
\floatname{algorithm}{Subroutine}
 \caption{\textsc{Special Case 1.3}}
\begin{algorithmic}[1]
 \Statex\textbf{Requirement:} There is an ancestor $u$ of $\hat v_c$ that is a reticulation and not an ancestor of $a$ or $b$.
        \Statel{Let $p_c$ be a parent of $\hat v_c$ that is a descendant of $u$}
        \Statel{Construct $\cN'$ from $\cN$ by deleting the arc $(p_c,\hat v_c)$.}
        \Statel{Return \textsc{TreeDetection$(\cN', \cT)$}.}
 \end{algorithmic}
\end{algorithm}
Since $u$ is an ancestor of $\hat v_c$, but not an ancestor of $a$ or $b$, we must have $\hat v_c=v_c$, and so $u$ must be verified by some leaf $\ell\neq c$. Since $u$ is not an ancestor of $a$ or $b$, it follows that $\ell\not\in \{a, b\}$. If an embedding of $\cT$ uses the arc $(p_c,\hat v_c)$, then it must display the rooted triple $c\ell|a$; a contradiction. Thus there is an embedding of $\cT$ in $\cN$ if and only if there is an embedding of $\cT$ in $\cN'$.

\begin{algorithm}[H]
\floatname{algorithm}{Subroutine}
 \caption{\textsc{Special Case 1.4}}
\begin{algorithmic}[1]
 \Statex\textbf{Requirement:} The leaf $c$ does not verify $\rho_{ab}$.
        \Statel{Let $p_c$ be a parent of $\hat v_c$ that is not a descendant of $\rho_{ab}$.}
        \Statel{Construct $\cN'$ from $\cN$ by deleting the arc $(p_c,\hat v_c)$.}
        \Statel{Return \textsc{TreeDetection$(\cN', \cT)$}.}
 \end{algorithmic}
\end{algorithm}
If $c$ does not verify $\rho_{ab}$, then there is a (directed) path $P$ from an ancestor of $\rho_{ab}$ to $\hat v_c$ not via $\rho_{ab}$, and $p_c$ is on this path. If we are in Situation~2 or Situation~3, then there is no embedding of $\cT$ in $\cN$ that uses arc $(p_c,\hat v_c)$. Otherwise, the embedding of $\cT$ displays the rooted triple $a\ell|c$, where $\ell$ is a leaf verifying $u$; a contradiction. If we are in Situation~1 and there is an embedding of $\cT$ that uses arc $(p_c,\hat v_c)$, then in the embedding the last common ancestor of $a$ and $c$ is an ancestor of $\rho_{ab}$, so (as above) we could adjust the embedding by resolving $v_a,v_b$ either way and it would still be valid. By (SC1.3), there is an unique path $Q$ from $\rho_{ab}$ to $\hat v_c$. Now $Q$ cannot use vertex $p_{ab}$ and vertex $q_{ab}$, so we can resolve $v_a$ or $v_b$ towards the ancestor not on $Q$ and $\hat v_c$ towards its parent on $Q$, and still have a valid embedding, but one not using arc $(p_c,\hat v_c)$. Thus  there is an embedding of $\cT$ in $\cN$ if and only if there is an embedding of $\cT$ in $\cN'$.

To complete the analysis of the first special case, we present that last of its subroutines and justify the actions in Table~\ref{singlectable}.

\begin{algorithm}[H]\label{SC1.5}
\floatname{algorithm}{Subroutine}
 \caption{\textsc{Special Case 1.5}}
\begin{algorithmic}[1]
 \Statex\textbf{Requirement:} Easy Cases 1--7 and Special Cases 1.1--1.4 do not apply. The sibling of the parent of $a$ and $b$ in $\cT$ is a single leaf $c$.
        \Statel{Determine which of the entries in Table~\ref{singlectable} applies.}
        \Statel{Construct $\cN'$ from $\cN$ by deleting the arc listed and set $\cT'=\cT$, or construct $\cN'=\cN\backslash b$ and $\cT'=\cT\backslash b$ if $b$ is listed instead of an arc.}
        \Statel{Return \textsc{TreeDetection$(\cN', \cT')$}.}
 \end{algorithmic}
\end{algorithm}

If an embedding of $\cT$ in $\cN$ uses the arc $(p_c,\hat v_c)$ and $p'_c$ is on the path $P_1$ or $P_2$, then $v_a$ and $v_b$ will have to resolve towards $q_a$ and $q_b$. Otherwise, the embedding does not have a vertex whose leaf descendants are just $a$ and $b$. Hence, if each of $p'_c$ and $q'_c$ lie on $P_1$ or $P_2$, then we can delete $(p_a, v_a)$. Similarly, if each of $p'_c$ and $q'_c$ lie on $P_3$ or $P_4$, then we can delete $(q_a, v_a)$.

If we are in Situation~1, and $p'_c$ lies on $P_1$ or $P_2$ and $q'_c$ lies on $P_3$ or $P_4$, then, in an embedding of $\cT$ in $\cN$, the last common ancestor of $a$ and $c$ would be $\rho_{ab}$ which ever way we resolved $\hat v_c$. Thus, in any of these cases, there is no path to a leaf other than $a$, $b$, and $c$ from $\rho_{ab}$ and so, if an embedding exists, one exists with $\hat v_c$ resolved each way (and $v_a$ and $v_b$ resolved accordingly). Thus, we can delete $(p_c, \hat v_c)$. Similar argument holds if we are in Situation~1, and $p'_c$ lies on $P_3$ or $P_4$ and $q'_c$ lies on $P_1$ or $P_2$.

If we are in Situations~2 or~3, and $p'_c$ lies on $P_3$ or $P_4$, no embedding of $\cT$ can resolve $\hat v_c$ towards $p_c$. This follows as then $v_a$ and $v_b$ would have to resolve towards $p_a$ and $p_b$ and, if $\ell_u$ is a leaf that verifies $u$ (note $\ell_u\neq c$), then the embedding of $\cT$ in $\cN$ would display the rooted triple $a\ell_u|c$; a contradiction. So we can delete arc $(p_c,\hat v_c)$. Similarly, if we are in Situations~2 or~3, and $q'_c$ lies on $P_3$ or $P_4$, then we can delete arc $(q_c, \hat v_c)$.

Now suppose $p'_c$ is an ancestor of $p_{ab}$ (denoted $>p_{ab}$ in the table). Then, if an embedding of $\cT$ in $\cN$ uses the arc $(p_c,\hat v_c)$ and arc $(q_a,v_a)$, the last common ancestor of $a$ and $c$ in the embedding would be an ancestor of $p'_c$. Thus there is no path to a leaf other than $a$, $b$, and $c$ from $p'_c$ and so an embedding exists with  $v_a$ and $v_b$ resolved towards $p_a$ and $p_b$, respectively. Similarly, if $p'_c$ is an ancestor of $q_{ab}$ and an embedding of $\cT$ in $\cN$ exists using arcs $(p_c, \hat v_c)$ and $(p_a, v_a)$, then one exists with $v_a$ and $v_b$ resolved towards $q_a$ and $q_b$, respectively. A symmetric argument holds with $q'_c$ instead of $p'_c$.

Thus if $p'_c$ is an ancestor of $p_{ab}$ and $q'_c$ lies on $P_3$ or $P_4$ (or vice versa), then we can delete arc $(q_a, v_a)$ and the resulting phylogenetic network on $X$ displays $\cT$ if and only if $\cN$ displays $\cT$. Likewise, if $p'_c$ is an ancestor of $q_{ab}$ and $q'_c$ lies on $P_1$ or $P_2$, then we can delete  arc $(p_a, v_a)$. Also, if both $p'_c$ and $q'_c$ are ancestors of $p_{ab}$, we can delete arc $(q_a,v_a)$, and if both are ancestors of $q_{ab}$, we can delete arc $(p_a,v_a)$. In each of the last three subcases, the resulting phylogenetic network on $X$ displays $\cT$ if and only if $\cN$ displays $\cT$.

If $p'_c$ is an ancestor of $p_{ab}$ and $q'_c$ lies on $P_1$ or $P_2$, and an embedding of $\cT$ in $\cN$ exists that resolves $\hat v_c$ towards $q_c$, then in this embedding $v_a$ and $v_b$ must resolve towards $q_a$ and $q_b$, respectively, in which case, in the embedding, the last common ancestor of $a$ and $c$ is $\rho_{ab}$ in Situation~1, and an ancestor of $u$ in Situations~2 and~3. In Situation~1, we could therefore resolve $\hat v_c$ towards $p_c$ without changing the topology of the embedded tree. In Situations~2 and~3, if $\ell_u\neq c$, where $\ell_u$ is a leaf that verifies $u$, then the embedding displays the rooted triple $c\ell_u|a$; a contradiction. Thus $\ell_u=c$, and therefore $p'_c$ lies on the path from $u$ to $p_{ab}$. We can therefore resolve $\hat v_c$ towards $p_c$ without changing the topology of the embedded tree. It follows that in either case, we can delete arc $(q_c, \hat v_c)$ and the resulting phylogenetic network on $X$ displays $\cT$ if and only if $\cN$ displays $\cT$. A similar argument holds if we interchange the roles of $p'_c$ and $q'_c$, in which case, we delete the arc $(p_c, \hat{v}_c)$. A symmetric argument holds in Situation~1 if $p'_c$ is an ancestor of $q_{ab}$ and $q'_c$ lies on $P_3$ or $P_4$ (or vice versa).

The remaining subcase is that $p'_c$ is an ancestor of $p_{ab}$ and $q'_c$ is an ancestor of $q_{ab}$, or vice versa. We now show that $\cN\backslash b$ displays $\cT\backslash b$ if and only if $\cN$ displays $\cT$.
In any embedding of $\cT\backslash b$ in $\cN\backslash b$, there is a vertex whose leaf descendants are precisely $a$ and $c$. Also, in the embedding, the peak of the embedded path from $a$ to $c$ must be an ancestor of either $p_{ab}$ or $q_{ab}$, since there is no (directed) path from either to $c$. If this peak is an ancestor of $p_{ab}$, then we can extend this embedding to an embedding of $\cT$ in $\cN$ by adding arcs $(p_b, v_b)$ and $(v_b, b)$, while if it is an ancestor of $q_{ab}$, then we extend the embedding to such an embedding by adding arcs $(q_b, v_b)$ and $(v_b, b)$. These must still be valid embeddings since there is a tree path from $p_{ab}$ to $p_b$, and from $q_{ab}$ to $q_b$, so in the embedding $b$ is in the correct location. Conversely, any embedding of $\cT$ in $\cN$ gives rise to an embedding of $\cT\backslash b$ in $\cN\backslash b$ by deleting $b$ and its incoming arc. This completes the justification of Table~\ref{singlectable}.

Finally, we describe the subroutines for the special case when the sibling of the parent of $a$ and $b$ in $\cT$ is the parent of a second cherry $\{c, d\}$.

\begin{algorithm}[H]
\floatname{algorithm}{Subroutine}
 \caption{\textsc{Special Case 2.1}}
\begin{algorithmic}[1]
 \Statex\textbf{Requirement:} The vertices $\rho_{cd}$ and $\rho_{ab}$ are non-comparable.
        \Statel{Construct $\cN'$ from $\cN$ by deleting the arc $(p_a, v_a)$.}
                \Statel{Return \textsc{TreeDetection$(\cN', \cT)$}.}
 \end{algorithmic}
\end{algorithm}

If $\rho_{cd}$ and $\rho_{ab}$ are non-comparable, then, in any embedding of $\cT$ in $\cN$, the last common ancestor of $a$ and $c$ is an ancestor of $\rho_{ab}$, and so there are no paths in the embedding from $\rho_{ab}$ to any leaf $\ell$ other than $a$ and $b$; otherwise, the embedding displays the rooted triple $a\ell|c$. Thus, we can resolve $v_a$ either way, and the resulting phylogenetic network on $X$ displays $\cT$ if and only if $\cN$ displays $\cT$.

\begin{algorithm}[H]
\floatname{algorithm}{Subroutine}
 \caption{\textsc{Special Case 2.2}}
\begin{algorithmic}[1]
 \Statex\textbf{Requirement:} The vertex $\rho_{cd}$ is not on one of the paths from $\rho_{ab}$ to either $v_a$ or $v_b$.
        \Statel{Construct $\cN'$ from $\cN$ by deleting the arc $(p_c, v_c)$.}
        \Statel{Return \textsc{TreeDetection$(\cN', \cT)$}.}
 \end{algorithmic}
\end{algorithm}

If $\rho_{cd}$ is not on one of the paths from $\rho_{ab}$ to either $v_a$ or $v_b$, then, in any embedding of $\cT$ in $\cN$, the last common ancestor of $a$ and $c$ is an ancestor of $\rho_{cd}$, and so there are no paths in the embedding from $\rho_{cd}$ to a leaf $\ell$ other than $c$ and $d$; otherwise the embedding displays the rooted triple $c\ell|a$. Therefore we can resolve $v_c$ either way, and the resulting phylogenetic network on $X$ displays $\cT$ if and only if $\cN$ displays $cT$.

\begin{algorithm}[H]
\floatname{algorithm}{Subroutine}
 \caption{\textsc{Special Case 2.3}}
\begin{algorithmic}[1]
 \Statex\textbf{Requirement:} There is a reticulation $u$ that is an ancestor of $p_{cd}$ or $q_{cd}$ but it is not an ancestor of $a$ or $b$.
	\If{$u$ is an ancestor of $p_{cd}$}
		\Statel{Construct $\cN'$ from $\cN$ by deleting the arc $(p_c, v_c)$.}
	\Else
	        \Statel{Construct $\cN'$ from $\cN$ by deleting the arc $(q_c, v_c)$.}
	       \EndIf
        \Statel{Return \textsc{TreeDetection$(\cN', \cT)$}.}
 \end{algorithmic}
\end{algorithm}

Reticulation $u$ must be verified by some leaf $\ell\not\in \{a, b, c, d\}$ since $u$ is a descendant of $\rho_{cd}$ and an ancestor of $v_c$ and $v_d$, but not an ancestor of either $a$ or $b$. Thus any embedding that resolves $v_c$ towards $u$ would display the rooted triple $c\ell|a$; a contradiction.

We now justify the final actions in Table~\ref{cherryctable}.

\begin{algorithm}[H]\label{SC2.4}
\floatname{algorithm}{Subroutine}
 \caption{\textsc{Special Case 2.4}}
\begin{algorithmic}[1]
 \Statex\textbf{Requirement:} Easy Cases 1--7 and Special Cases 2.1--2.3 do not apply. The sibling of the parent of $a$ and $b$ in $\cT$ is the parent of a cherry $\{c, d\}$.
        \Statel{Determine which of the entries in Table~\ref{cherryctable} applies.}
        \Statel{Construct $\cN'$ from $\cN$ by deleting the arc listed and set $\cT'=\cT$, or construct $\cN'=\cN\backslash b$ and $\cT'=\cT\backslash b$ if $b$ is listed instead of an arc.}
        \Statel{Return \textsc{TreeDetection$(\cN', \cT')$}.}
 \end{algorithmic}
\end{algorithm}

If an embedding of $\cT$ in $\cN$ uses the arc $(p_c, v_c)$ and $p'_{cd}$ is on the path $P_1$ or $P_2$, then $v_a$ and $v_b$ will have to resolve towards $q_a$ and $q_b$, respectively, so that the embedding contains a vertex whose leaf descendants are $a$ and $b$. Likewise, if $p'_{cd}$ is an ancestor of $p_{ab}$ and $p'_{cd}=p_{cd}$, so that $c$ and $d$ do not meet before reaching an ancestor of $a$, then an embedding of $\cT$ in $\cN$ that uses arc $(p_c, v_c)$ implies $v_a$ and $v_b$ will have to resolve towards $q_a$ and $q_b$, respectively; otherwise, one of the rooted triples $ac|d$ and $ad|c$ would be displayed by the embedding. The same outcomes hold if replace $(p_c, v_c$ and $p'_{cd}$ with $(q_c, v_c)$ and $q'_{cd}$, respectively, in the hypotheses. Thus if both $p'_{cd}$ and $q'_{cd}$ are of Type~1, we can delete $(p_a, v_a)$ since, either way $v_c$ resolves, the embedding cannot use that arc, and the resulting phylogenetic network on $X$ displays $\cT$ if and only if $\cN$ displays $\cT$. Similarly, if both $p'_{cd}$ and $q'_{cd}$ are of Type~2, we can delete $(q_a,v_a)$.

If we are in Situation~1 and, without loss of generality, $p'_{cd}$ is Type~1 and $q'_{cd}$ is Type~2, then, in an embedding of $\cT$ in $\cN$, the last common ancestor of $a$ and $c$ would be $\rho_{ab}$ which ever way we resolved $v_c$. Thus, in any such embedding, there is no path to a leaf other than $a$, $b$, $c$, and $d$ from $\rho_{ab}$ and so, if an embedding exists, one exists with $v_c$ resolved each way (and $v_a$ and $v_b$ resolved accordingly). Thus we can delete $(p_c, v_c)$, and the resulting phylogenetic network on $X$ displays $\cT$ if and only if $\cN$ displays $\cT$.

If we are in Situations~2 or~3, and $p'_{cd}$ is of Type~2, then no embedding of $\cT$ in $\cN$ can resolve $ v_c$ towards $p_c$. This follows as then $v_a$ and $v_b$ would have to resolve towards $p_a$ and $p_b$, respectively. But then if $\ell_u$ is a leaf that verifies $u$, and noting that $\ell_u\not\in \{c, d\}$, the embedding would display the rooted triple $a\ell_u|c$; a contradiction. So we can delete arc $(p_c, v_c)$, and the resulting phylogenetic network on $X$ displays $\cT$ if and only if $\cN$ displays $\cT$. Similarly, if $q'_{cd}$ is of Type~2, we can delete arc $(q_c, v_c)$.

Now suppose $p'_{cd}$ is of Type~3. Then if an embedding of $\cT$ in $\cN$ uses the arc $(p_c, v_c)$ and arc $(q_a,v_a)$, it follows that the last common ancestor of $a$ and $c$ in this embedding would be an ancestor of $p'_{cd}$. Thus in such an embedding there is no path to a leaf other than $a$, $b$, $c$, and $d$ from $p'_{cd}$, and so an embedding also exists with $v_a$ and $v_b$ resolved towards $p_a$ and $p_b$. Similarly, if $p'_{cd}$ is of Type~4 and an embedding of $\cT$ in $\cN$ exists using arcs $(p_c, v_c)$ and $(p_a, v_a)$, then one exists with $v_a$ and $v_b$ resolved towards $q_a$ and $q_b$. A symmetric argument holds with $q'_{cd}$ instead of $p'_{cd}$.

Thus if $p'_{cd}$ is of Type~3 and $q'_c$ is of Type~2 (or vice versa), we can delete arc $(q_a, v_a)$, and the resulting phylogenetic network on $X$ displays $\cT$ if and only if $\cN$ displays $\cT$. Likewise, if $p'_{cd}$ is of Type~4 and $q'_{cd}$ is of Type~1 (or vice versa), we can delete arc $(p_a,v_a)$. Also, if both $p'_{cd}$ and $q'_{cd}$ are of Type~3, we can delete arc $(q_a, v_a)$, and if both are of Type~4, we can delete arc $(p_a, v_a)$. In each of the last three subcases, the resulting phylogenetic network displays $\cT$ if and only if $\cN$ displays $\cT$.

Suppose that $p'_{cd}$ is of Type~3 and $q'_{cd}$ is of Type~1, and that an embedding of $\cT$ in $\cN$ exists that resolves $v_c$ and $v_d$ towards $q_c$ and $q_d$, respectively. Then, in this embedding, $v_a$ and $v_b$ must resolve towards $q_a$ and $q_b$, respectively, and the last common ancestor of $a$ and $c$ is $\rho_{ab}$ in Situation~1, and an ancestor of $u$ in Situation~2 and~3. In Situation~1, we could therefore resolve $v_c$ and $v_d$ towards $p_c$ and $p_d$ without changing the topology of the embedded tree. In Situations~2 and~3, if $l_u\not\in \{c, d\}$, where $\ell_u$ is a leaf that verifies $u$, then the embedding displays the rooted triple $c\ell_u|a$; a contradiction. Thus, in Situations~2 and~3, we may assume $\ell_u\in \{c, d\}$, and therefore $p'_{cd}$ lies on the path from $u$ to $p_{ab}$. We can therefore resolve $v_c$ to $p_c$ without changing the topology of the embedding of $\cT$. Regardless of the situation, we can delete arc $(q_c, v_c)$, and the resulting phylogenetic network on $X$ displays $\cT$ if and only if $\cN$ displays $\cT$. A similar argument holds if $p'_{cd}$ is of Type~1 and $q'_{cd}$ is of Type~3, in which case, we delete $(p_c, v_c)$. A symmetric argument holds in Situation~1 if $p'_{cd}$ is of Type~4 and $q'_{cd}$ is of Type~2 (or vice versa).

The remaining case is that $p'_{cd}$ is of Type 3 and $q'_{cd}$ is of Type~2 (or vice versa). An identical argument to that given to justify the entries in Table~\ref{singlectable} corrsponding to deleting the leaf $b$ in the first special case applies. In particular, this gives that $\cN\backslash b$ displays $\cT\backslash b$ if and only if $\cN$ displays $\cT$.

\section{Running Time}
\label{running}

In this section we consider the running time of the {\sc TreeDetection} algorithm. The input to the algorithm is a reticulation-visible network $\cN$ on a finite set $X$ and a rooted binary phylogenetic $X$-tree $\cT$. Traditionally running times in phylogenetics are given in terms of the size of the taxa set, $|X|$, although  the representation of a general phylogenetic network may be much larger that $O(|X|)$. However in the case of reticulation-visible networks, the number of vertices of the network is at most $8|X|-7$, and so the input in this case is of size $O(|X|)$ (see Theorem~\ref{sharp}).

Clearly determining if $|X|=1$ can be done in constant time, and finding a cherry $a,b$ in $\cT$ can be done in linear time. The main body of the algorithm then boils down to determining which of several cases occurs, and then running a specific subroutine. We first show that we can determine which case occurs in $O(|X|^2)$ steps. To check the Easy Cases we need to know the vertices $v_a,v_b$ and $\rho_{ab}$, and their ancestors and descendants. For each vertex of $\cN$ we prepare a list of:
\begin{itemize}
\item all ancestors
\item all descendants
\item all those descendants that verify the vertex
\item all those descendants reachable by a tree-path.
\end{itemize} We also order the vertices by distance of the vertex from the root; since there are less than $8|X|$ vertices altogether, this can be done in $O(|X|^2)$ steps using a depth first search. We can then scan through the list in order to determine $v_a,v_b$ and $\rho_{ab}$, by checking whether each vertex satisfies the conditions. Further scanning through the list and applying simple checks at each vertex can determine which, if any, of the Easy Cases occurs, and since we only need compare at most two vertices (EC7), this can all be achieved in $O(|X|^2)$ steps.

Depending on whether the sibling of the cherry $a,b$ is a single leaf $c$ or another cherry $c,d$, we must then check which of several Special Cases occurs. However, once we have identified $\hat v_c$, or $\rho_{cd},p_{cd},q_{cd}$, by looking through the lists above, it is again a matter of checking whether vertices with certain straightforward conditions on their ancestors and descendants exist or not, and this can be done in $O(|X|^2)$ steps.

Finally, each of the subroutines either returns \textbf{No} or modifies $\cN$ and possibly $\cT$ before recursively calling {\sc TreeDetection}. Determining whether to return \textbf{No} or make the modifications can be  accomplished in $O(|X|^2)$ steps, since they again require at most checking the existence of vertices with specific ancestor or descendant conditions, possibly including reachability by a tree-path, but since the network is of linear size, this is can be done, and then any modifications of $\cN$ and  $\cT$ are only by deleting a constant number of arcs or vertices. This analysis includes determining which row of the final tables is appropriate, since we can identify the vertices $p'_c,q'_c$, or $p'_{cd},q'_{cd}$ by scanning through the vertex list and checking straightforward ancestor/descendant conditions, and then determining which of the rows is applicable is again checking ancestor and descendant conditions on these vertices, e.g. we can tell if a vertex is on a path from $p_{ab}$ to $p_a$ by checking if it is a descendant of $p_{ab}$ and an ancestor of $p_a$. 

Thus the entire algorithm breaks down into a constant number of checks, each of which can be accomplished in $O(|X|^2)$ steps, followed by a small modification to $\cN$ and possibly $\cT$, and a recursive call to {\sc TreeDetection} on an input that is smaller in either the number of reticulations or the size of the leaf set. Since the number of reticulations is linear in $|X|$, there can be at most $O(|X|)$ recursive calls, and so the entire algorithm completes in $O(|X|^3)$ steps.

\section{Sharp Bounds}
\label{visible}




In this section, we establish Theorem~\ref{sharp}.

\begin{proof}[Proof of Theorem~\ref{sharp}]
Let $\cN$ be a reticulation-visible network on $X$ with $n$ vertices in total and $r$ reticulations. Let $m=|X|$. We first show that
\begin{align}
n\le 8m-7
\label{ineq1}
\end{align}
and
\begin{align}
r\le 3m-3.
\label{ineq2}
\end{align}
The proof of these two inequalities is by induction on $m$. If $m=1$, then $\cN$ consists of a single vertex, and~(\ref{ineq1}) and~(\ref{ineq2}) hold. Suppose that $m\ge 2$, and that~(\ref{ineq1}) and~(\ref{ineq2}) hold for all reticulation-visible networks with fewer leaves.

First assume that $\cN$ has a cherry $\{a, b\}$. Let $\cN'$ be the network obtained from $\cN$ by reducing $\{a, b\}$. Since every vertex in $\cN$ is visible, it follows that every vertex in $\cN'$ is visible, and so $\cN'$ is a reticulation-visible network. Therefore, as $\cN'$ has $n-2$ vertices, $r$ reticulations, and $m-1$ leaves, it follows by the induction assumption that $n-2\leq 8(m-1)-7$ and $r\le 3(m-1)-3$, so
$$n\leq 8m-13\leq 8m-7$$
and
$$r\leq 3m-6\leq 3m-3.$$
Thus (\ref{ineq1}) and~(\ref{ineq2}) hold.

Now assume that $\cN$ does not contain a cherry. Let $v$ be a reticulation in $\cN$ such that amongst all reticulations in $\cN$ it is at maximum distance from the root $\rho$. Let $P$ be a path from $\rho$ to $v$ that realises this maximum distance. By maximality and the assumption that $\cN$ has no cherry, the child vertex of $v$ is a leaf, $\ell$ say. Also, note that, as $\cN$ is reticulation visible, neither parent of $v$ is a reticulation. Let $\cN'$ be the network obtained from $\cN$ by deleting the vertices $v$ and $\ell$ and their incident arcs, contracting any resulting degree-two vertices, and then replacing any parallel arcs with a single arc and contracting any degree-two vertices resulting from this replacement. Since the final step in this process could not have created any further parallel arcs, it is easily seen that $\cN'$ is a phylogenetic network on $X-\{\ell\}$. Furthermore, in the process of obtaining $\cN$ from $\cN'$, we initially lose $1$ reticulation and $4$ vertices in total. Additionally, at most two pairs of parallel arcs are replaced with a single arc. Each such replacement, loses $1$ reticulation and $2$ vertices in total. Thus, if $n'$ and $r'$ denotes the total number of vertices and the number of reticulations in $\cN'$, we have
\begin{align}
n-8\le n'\le n-4
\label{eqn1}
\end{align}
and
\begin{align}
r-3\le r'\le r-1.
\label{eqn2}
\end{align}

We next show that either $\cN'$ or a phylogenetic network obtained from $\cN$ in an analogous way is reticulation visible, thereby obtaining a phylogenetic network that satisfies the induction assumptions. Let $p$ and $q$ denote the parents of $v$. If $\ell$ does not verify the visibility of any reticulation other than $v$ in $\cN$, then $\cN'$ is reticulation visible. Therefore, suppose that $\ell$ also verifies the visibility of a reticulation $w$ in $\cN$, where $w\neq v$. Without loss of generality, we may assume that $p$ is the last parent of $v$ on $P$. If $p$ is an ancestor of $q$, then $P$ is not a path of maximum distance from $\rho$ to $v$, so either $p$ is a descendant of $q$ or $p$ is non-comparable to $q$. In either case, this implies that $p$ has a child vertex, not equal $v$, that is not an ancestor of $v$. Let $v'$ denote this child vertex. By the maximality of $P$ and the assumption that $\cN$ has no cherries, $v'$ is either a leaf or a reticulation. If $v'$ is a leaf, then, as every path in $\cN$ from $\rho$ to $\ell$ passes through $w$, every path in $\cN$ from $\rho$ to $v'$ also passes through $w$. It now follows that if $v'$ is a leaf, $\cN'$ is reticulation visible. Therefore suppose that $v'$  is a reticulation. By the maximality of $P$ and the assumption that $\cN$ has no cherries, the child of $v'$ is a leaf, $\ell'$ say. Let $\cN'_1$ be the network obtained from $\cN$ by deleting $v'$ and $\ell'$ and their incident arcs, contracting any resulting degree-two vertices, and then replacing any parallel arcs with a single arc and contracting any degree-two vertex resulting from this replacement. As above, if $n'_1$ and $r'_1$ denote the total number of vertices and the number of reticulations in $\cN'_1$, then
\begin{align*}
n-8\le n'_1\le n-4
\end{align*}
and
\begin{align*}
r-3\le r'_1\le r-1.
\end{align*}

If $\ell'$ does not verify the visibility of any reticulation other than $v'$ in $\cN$, then, instead of applying the inductive step to $v$, apply the inductive step to $v'$. In particular, $\cN'_1$ is reticulation visible and therefore satisfies the induction assumptions. Thus we may assume that $\ell'$ also verifies the visibility of a reticulation $w'$ in $\cN$, where $w'\neq v'$. Now every path in $\cN$ from $\rho$ to $\ell$ passes through $w$, in particular, every path in $\cN$ from $\rho$ to $\ell$ using the arc directed into $p$ passes through $w$. In turn, this implies that every path in $\cN$ from $\rho$ to $\ell'$ using the arc directed into $p$ passes through $w$. We deduce that either $w$ is an ancestor of $w'$ or $w'$ is an ancestor of $w$ in $\cN$. If $w$ is an ancestor $w'$ in $\cN$, then every path in $\cN$ from $\rho$ to $\ell'$ passes through $w$, and so $\cN'$ is reticulation visible, in which case, $\cN'$ satisfies the induction assumptions. On the other hand, if $w'$ is an ancestor of $w$ in $\cN$, then every path in $\cN$ from $\rho$ to $\ell$ passes through $w'$, and so $\cN'_1$ is reticulation visible, in which case, $\cN'_1$ satisfies the induction assumptions.

Without loss of generality, we may assume that the induction assumptions hold for $\cN'$. By induction, and (\ref{eqn1}) and (\ref{eqn2}), it follows that
$$n-8\le n'\le 8(m-1)-7=8m-15,$$
so $n\le 8m-7$, and
$$r-3\le r'\le 3(m-1)-3=3m-6,$$
so $r\le 3m-3$. Hence~(\ref{ineq1}) and~(\ref{ineq2}) hold.

\begin{figure}
\center
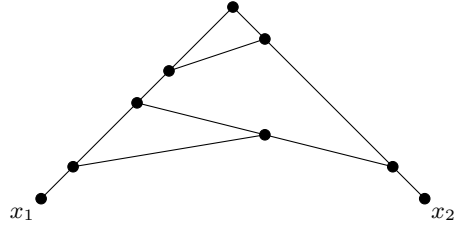
\caption{A reticulation-visible network with $2$~leaves, $3$~reticulations, and $9$~vertices in total.}
\label{2sharp}
\end{figure}

To show that the upper bounds~(\ref{ineq1}) and~(\ref{ineq2}) are sharp for all integers $m\ge 1$, consider the reticulation-visible network shown in Fig.~\ref{2sharp}. Here the bound is tight for when $m=2$.
For all integers $m\ge 3$, an analogous example can be constructed by replacing a leaf with a complete copy of this network. The resulting network is reticulation-visible, and the number of leaves has increased by~$1$, the number of reticulations has increased by~$3$, and the total number of vertices has increased by~$8$, thereby retaining the tightness of the bound. Of course, if $m=1$, then $\cN$ consists of a single vertex and the bounds hold exactly.


\end{proof}

\end{document}

%% file: display.pdf_t
\begin{picture}(0,0)%
\includegraphics{display.pdf}%
\end{picture}%
\setlength{\unitlength}{3522sp}%
\begingroup\makeatletter\ifx\SetFigFont\undefined%
\gdef\SetFigFont#1#2#3#4#5{%
  \reset@font\fontsize{#1}{#2pt}%
  \fontfamily{#3}\fontseries{#4}\fontshape{#5}%
  \selectfont}%
\fi\endgroup%
\begin{picture}(6420,1999)(841,-3140)
\put(2251,-3076){\makebox(0,0)[b]{\smash{{\SetFigFont{9}{10.8}{\rmdefault}{\mddefault}{\updefault}$\cN$}}}}
\put(2431,-2671){\makebox(0,0)[rb]{\smash{{\SetFigFont{9}{10.8}{\rmdefault}{\mddefault}{\updefault}$x_3$}}}}
\put(2971,-2671){\makebox(0,0)[lb]{\smash{{\SetFigFont{9}{10.8}{\rmdefault}{\mddefault}{\updefault}$x_4$}}}}
\put(4456,-2671){\makebox(0,0)[rb]{\smash{{\SetFigFont{9}{10.8}{\rmdefault}{\mddefault}{\updefault}$x_1$}}}}
\put(5356,-2671){\makebox(0,0)[rb]{\smash{{\SetFigFont{9}{10.8}{\rmdefault}{\mddefault}{\updefault}$x_2$}}}}
\put(7246,-2671){\makebox(0,0)[lb]{\smash{{\SetFigFont{9}{10.8}{\rmdefault}{\mddefault}{\updefault}$x_5$}}}}
\put(5851,-3076){\makebox(0,0)[b]{\smash{{\SetFigFont{9}{10.8}{\rmdefault}{\mddefault}{\updefault}$\cT$}}}}
\put(6031,-2671){\makebox(0,0)[rb]{\smash{{\SetFigFont{9}{10.8}{\rmdefault}{\mddefault}{\updefault}$x_3$}}}}
\put(6571,-2671){\makebox(0,0)[lb]{\smash{{\SetFigFont{9}{10.8}{\rmdefault}{\mddefault}{\updefault}$x_4$}}}}
\put(856,-2671){\makebox(0,0)[rb]{\smash{{\SetFigFont{9}{10.8}{\rmdefault}{\mddefault}{\updefault}$x_1$}}}}
\put(1756,-2671){\makebox(0,0)[rb]{\smash{{\SetFigFont{9}{10.8}{\rmdefault}{\mddefault}{\updefault}$x_2$}}}}
\put(3646,-2671){\makebox(0,0)[lb]{\smash{{\SetFigFont{9}{10.8}{\rmdefault}{\mddefault}{\updefault}$x_5$}}}}
\end{picture}%

%% file: situations.pdf_t
\begin{picture}(0,0)%
\includegraphics{situations.pdf}%
\end{picture}%
\setlength{\unitlength}{3522sp}%
\begingroup\makeatletter\ifx\SetFigFont\undefined%
\gdef\SetFigFont#1#2#3#4#5{%
  \reset@font\fontsize{#1}{#2pt}%
  \fontfamily{#3}\fontseries{#4}\fontshape{#5}%
  \selectfont}%
\fi\endgroup%
\begin{picture}(5520,7125)(1291,-6295)
\put(2251,-3076){\makebox(0,0)[b]{\smash{{\SetFigFont{9}{10.8}{\rmdefault}{\mddefault}{\updefault}(i) Situation 1}}}}
\put(2206,-691){\makebox(0,0)[rb]{\smash{{\SetFigFont{9}{10.8}{\rmdefault}{\mddefault}{\updefault}$\rho_{ab}$}}}}
\put(5806,659){\makebox(0,0)[rb]{\smash{{\SetFigFont{9}{10.8}{\rmdefault}{\mddefault}{\updefault}$\rho_{ab}$}}}}
\put(1531,-2446){\makebox(0,0)[rb]{\smash{{\SetFigFont{9}{10.8}{\rmdefault}{\mddefault}{\updefault}$v_a$}}}}
\put(1531,-2671){\makebox(0,0)[rb]{\smash{{\SetFigFont{9}{10.8}{\rmdefault}{\mddefault}{\updefault}$a$}}}}
\put(1306,-2041){\makebox(0,0)[rb]{\smash{{\SetFigFont{9}{10.8}{\rmdefault}{\mddefault}{\updefault}$p_a$}}}}
\put(1756,-2041){\makebox(0,0)[rb]{\smash{{\SetFigFont{9}{10.8}{\rmdefault}{\mddefault}{\updefault}$q_a$}}}}
\put(2881,-2671){\makebox(0,0)[rb]{\smash{{\SetFigFont{9}{10.8}{\rmdefault}{\mddefault}{\updefault}$b$}}}}
\put(2881,-2446){\makebox(0,0)[rb]{\smash{{\SetFigFont{9}{10.8}{\rmdefault}{\mddefault}{\updefault}$v_b$}}}}
\put(3196,-2041){\makebox(0,0)[lb]{\smash{{\SetFigFont{9}{10.8}{\rmdefault}{\mddefault}{\updefault}$q_b$}}}}
\put(2791,-2041){\makebox(0,0)[lb]{\smash{{\SetFigFont{9}{10.8}{\rmdefault}{\mddefault}{\updefault}$p_b$}}}}
\put(1531,-1366){\makebox(0,0)[rb]{\smash{{\SetFigFont{9}{10.8}{\rmdefault}{\mddefault}{\updefault}$p_{ab}$}}}}
\put(2971,-1366){\makebox(0,0)[lb]{\smash{{\SetFigFont{9}{10.8}{\rmdefault}{\mddefault}{\updefault}$q_{ab}$}}}}
\put(5851,-2626){\makebox(0,0)[b]{\smash{{\SetFigFont{9}{10.8}{\rmdefault}{\mddefault}{\updefault}(ii) Situation 2}}}}
\put(5131,-2221){\makebox(0,0)[rb]{\smash{{\SetFigFont{9}{10.8}{\rmdefault}{\mddefault}{\updefault}$a$}}}}
\put(5131,-1996){\makebox(0,0)[rb]{\smash{{\SetFigFont{9}{10.8}{\rmdefault}{\mddefault}{\updefault}$v_a$}}}}
\put(4906,-1591){\makebox(0,0)[rb]{\smash{{\SetFigFont{9}{10.8}{\rmdefault}{\mddefault}{\updefault}$p_a$}}}}
\put(5356,-1591){\makebox(0,0)[rb]{\smash{{\SetFigFont{9}{10.8}{\rmdefault}{\mddefault}{\updefault}$q_a$}}}}
\put(6481,-2221){\makebox(0,0)[rb]{\smash{{\SetFigFont{9}{10.8}{\rmdefault}{\mddefault}{\updefault}$b$}}}}
\put(6481,-1996){\makebox(0,0)[rb]{\smash{{\SetFigFont{9}{10.8}{\rmdefault}{\mddefault}{\updefault}$v_b$}}}}
\put(6391,-1591){\makebox(0,0)[lb]{\smash{{\SetFigFont{9}{10.8}{\rmdefault}{\mddefault}{\updefault}$p_b$}}}}
\put(6796,-1591){\makebox(0,0)[lb]{\smash{{\SetFigFont{9}{10.8}{\rmdefault}{\mddefault}{\updefault}$q_b$}}}}
\put(5131,-916){\makebox(0,0)[rb]{\smash{{\SetFigFont{9}{10.8}{\rmdefault}{\mddefault}{\updefault}$p_{ab}$}}}}
\put(5221,-871){\makebox(0,0)[lb]{\smash{{\SetFigFont{9}{10.8}{\rmdefault}{\mddefault}{\updefault}$u$}}}}
\put(6571,-916){\makebox(0,0)[lb]{\smash{{\SetFigFont{9}{10.8}{\rmdefault}{\mddefault}{\updefault}$q_{ab}$}}}}
\put(5806,-2941){\makebox(0,0)[rb]{\smash{{\SetFigFont{9}{10.8}{\rmdefault}{\mddefault}{\updefault}$\rho_{ab}$}}}}
\put(5851,-6226){\makebox(0,0)[b]{\smash{{\SetFigFont{9}{10.8}{\rmdefault}{\mddefault}{\updefault}(iii) Situation 3}}}}
\put(5131,-5821){\makebox(0,0)[rb]{\smash{{\SetFigFont{9}{10.8}{\rmdefault}{\mddefault}{\updefault}$a$}}}}
\put(5131,-5596){\makebox(0,0)[rb]{\smash{{\SetFigFont{9}{10.8}{\rmdefault}{\mddefault}{\updefault}$v_a$}}}}
\put(4906,-5191){\makebox(0,0)[rb]{\smash{{\SetFigFont{9}{10.8}{\rmdefault}{\mddefault}{\updefault}$q_a$}}}}
\put(5356,-5191){\makebox(0,0)[rb]{\smash{{\SetFigFont{9}{10.8}{\rmdefault}{\mddefault}{\updefault}$p_a$}}}}
\put(6481,-5821){\makebox(0,0)[rb]{\smash{{\SetFigFont{9}{10.8}{\rmdefault}{\mddefault}{\updefault}$b$}}}}
\put(6481,-5596){\makebox(0,0)[rb]{\smash{{\SetFigFont{9}{10.8}{\rmdefault}{\mddefault}{\updefault}$v_b$}}}}
\put(6391,-5191){\makebox(0,0)[lb]{\smash{{\SetFigFont{9}{10.8}{\rmdefault}{\mddefault}{\updefault}$p_b$}}}}
\put(6796,-5191){\makebox(0,0)[lb]{\smash{{\SetFigFont{9}{10.8}{\rmdefault}{\mddefault}{\updefault}$q_b$}}}}
\put(5806,-4966){\makebox(0,0)[rb]{\smash{{\SetFigFont{9}{10.8}{\rmdefault}{\mddefault}{\updefault}$p_{ab}$}}}}
\put(5896,-4921){\makebox(0,0)[lb]{\smash{{\SetFigFont{9}{10.8}{\rmdefault}{\mddefault}{\updefault}$u$}}}}
\put(5356,-4066){\makebox(0,0)[rb]{\smash{{\SetFigFont{9}{10.8}{\rmdefault}{\mddefault}{\updefault}$q_{ab}$}}}}
\end{picture}%

%% file: sharp.pdf_t
\begin{picture}(0,0)%
\includegraphics{sharp.pdf}%
\end{picture}%
\setlength{\unitlength}{3522sp}%
\begingroup\makeatletter\ifx\SetFigFont\undefined%
\gdef\SetFigFont#1#2#3#4#5{%
  \reset@font\fontsize{#1}{#2pt}%
  \fontfamily{#3}\fontseries{#4}\fontshape{#5}%
  \selectfont}%
\fi\endgroup%
\begin{picture}(2820,1603)(841,-2744)
\put(3646,-2671){\makebox(0,0)[lb]{\smash{{\SetFigFont{9}{10.8}{\rmdefault}{\mddefault}{\updefault}$x_2$}}}}
\put(856,-2671){\makebox(0,0)[rb]{\smash{{\SetFigFont{9}{10.8}{\rmdefault}{\mddefault}{\updefault}$x_1$}}}}
\end{picture}%